\documentclass{mfat}
\pagespan{152}{168}

\usepackage{mmap}
\usepackage[utf8]{inputenc}

\renewcommand{\theequation}{\arabic{section}.\arabic{equation}}

\theoremstyle{plain}
    \newtheorem{Pp}{Proposition}[section]
    \newtheorem{Thm}[Pp]{Theorem}
    \newtheorem{Lm}[Pp]{Lemma}
    
    \newtheorem*{H1}{Assumption (H1)}
    \newtheorem*{H2}{Assumption (H2)}
    \newtheorem*{H3}{Assumption (H3)}
    \newtheorem*{H4}{Assumption (H4)}
\theoremstyle{definition}
    \newtheorem*{Data}{Hypocoercivity data (D)}
    \newtheorem{Df}[Pp]{Definition}
    \newtheorem{Rm}[Pp]{Remark}
    \newtheorem*{Ass3}{Hypocoercivity assumptions (C1)--(C3)}

\def\D{D}
\def\H{H}
\def\d{\,\mathrm{d}}
\def\df{:=}

\begin{document}

\title[Hypocoercivity for the Langevin dynamics revisited]
      {Hilbert space hypocoercivity for the Langevin dynamics revisited}

\author{Martin Grothaus}
\address{Mathematics Department, University of Kaiserslautern,
P.O.Box 3049, 67653 Kaiserslau\-tern, Germany}
\email{grothaus@mathematik.uni-kl.de}

\author{Patrik Stilgenbauer}
\address{Mathematics Department, University of Kaiserslautern,
P.O.Box 3049, 67653 Kaiserslau\-tern, Germany}
\email{stilgenb@mathematik.uni-kl.de}

\date{25/09/2015}

\begin{abstract}
We provide a complete elaboration of the $L^2$-Hilbert space
hypocoer\-civity theorem for the degenerate Langevin dynamics
via studying the longtime behavior of the strongly continuous
contraction semigroup solving the associated \hbox{Kolmogorov} (backward)
equation as an abstract Cauchy problem. This hypocoercivity result
is proven in previous works before by Dolbeault, Mouhot and Schmeiser
in the corresponding dual Fokker-Planck framework, but without
including domain issues of the appearing operators.
In our elaboration, we include the domain issues and additionally
compute the rate of convergence in dependence of the damping coefficient.
Important statements for the complete elaboration are the m-dissipativity results
for the Langevin operator established by Conrad and the first named author
of this article as well as the essential selfadjointness results
for generalized Schr\"{o}dinger ope\-ra\-tors by Wielens or Bogachev, Krylov and R\"{o}ckner.
We emphasize that the chosen Kolmogorov approach is natural.
Indeed, techniques from the theory of (genera\-lized) Dirichlet forms
imply a stochastic representation of the Langevin semigroup
as the transition kernel of diffusion process which provides
a martingale solution to the Langevin equation. Hence
an interesting connection between the theory of hypocoercivity
and the theory of (generalized) Dirichlet forms is established besides.
\end{abstract}

\maketitle

\section{Introduction} \label{section_Introduction_Hypo}

In this article we are interested in studying the exponential decay to equilibrium of the classical Langevin dynamics.
The corresponding evolution equation is given by the following stochastic differential equation (SDE)
on $\mathbb{R}^{2d}$, $d \in \mathbb{N}$, as
\begin{align} \label{Langevin_in _Rd_chapter_hypocoercivity}
\mathrm{d}x_t &= \omega_t \, \mathrm{dt}, \\
\mathrm{d}\omega_t &= -\alpha \,\omega_t \, \mathrm{dt} -  \nabla \Psi(x_t)\, \mathrm{dt}
+ \sqrt{\frac{2\alpha}{\beta}}\,  \mathrm{d}W_t, \nonumber
\end{align}
where $\alpha, \beta \in (0,\infty)$, $\Psi \colon \mathbb{R}^{d} \to \mathbb{R}$ is a suitable potential function which needs to be specified later on and $W$ denotes a standard $d$-dimensional Brownian motion. For the national convenience below, we redefine the potential via setting
\begin{align*}
\Phi\df \beta \, \Psi.
\end{align*}
The Langevin equation \eqref{Langevin_in _Rd_chapter_hypocoercivity} describes the evolution of a particle, described by position  $(x_t)_{t \geq 0}$ and velocity coordinates $(\omega_t)_{t \geq 0}$, which is subject to friction, stochastic perturbation and an external forcing term $\nabla \Psi$, see \cite[Ch.~8]{Sch06} and \cite{CKW04} for the background. $\alpha > 0$ is called the damping coefficient. The Kolmogorov generator associated to \eqref{Langevin_in _Rd_chapter_hypocoercivity} is given at first formally by
\begin{align} \label{Eq_generator_Langevin_hypocoercivity}
L=\omega \cdot \nabla_x - \alpha ~ \omega \cdot \nabla_\omega - \frac{1}{\beta} \, \nabla_x \Phi \cdot \nabla_\omega + \frac{\alpha}{\beta}\, \Delta_\omega.
\end{align}
Here $\cdot$ or $(\cdot,\cdot)_{\text{euc}}$ denotes the standard Euclidean scalar product, $\nabla_x$ and $\nabla_\omega$ the usual gradient operators in $\mathbb{R}^{d}$ for the respective $x$- or $\omega$-direction and $\Delta_\omega$ is the Laplace-operator in $\mathbb{R}^{d}$ in the $\omega$-direction. We introduce the measure $\mu_{\Phi,\beta}$ as
\begin{align*}
\mu_{\Phi,\beta} = \frac{1}{\sqrt{2\pi \beta^{-1}}^{d}} ~ e^{- \Phi(x) - \beta \frac{\omega^2}{2}} \,\mathrm{d} x \otimes \mathrm{d} \omega = e^{- \Phi(x)} \,\mathrm{d} x \otimes  \nu_\beta.
\end{align*}
Above $\mathrm{d}x$\index{$\d x$} and $\mathrm{d}\omega$\index{$\d \omega$} denote the Lebesgue measure on $(\mathbb{R}^{d},\mathcal{B}(\mathbb{R}^d))$, $\omega^2\df \omega \cdot \omega$ and $\nu_\beta$ is the normalized Gaussian measure on $\mathbb{R}^{d}$\index{Gaussian measure} with mean $0$ and covariance matrix $\beta^{-1} I$. In case $\mu_{\Phi,\beta}$ is finite, it is up to normalization the canonical invariant measure or canonical stationary equilibrium distribution for the dynamics described by \eqref{Langevin_in _Rd_chapter_hypocoercivity}.


Due to the degenerate structure of the Langevin equation (i.e., the stochastic only acts in the velocity), studying the exponential decay to equilibrium is non-trivial and provides demanding mathematical challenges. Nevertheless, in the last decade many probabilistic and functional analytic tools are developed for studying the exponential longtime behavior of the Langevin dynamics or its associated Fokker-Planck evolution equation; see e.g.~\cite{Wu01}, \cite{MS02}, \cite{HeNi04}, \cite{HN05}, \cite{Her07}, \cite{BCG08}, \cite{Vil09}, \cite{DMS09}, \cite{Mon13}, \cite{Bau13} and \cite{DMS13}.

In the underlying article we are interested in applying functional analytic methods based on \textit{hypocoercivity}. Here the word \textit{hypocoercivity} addresses the study of the exponential convergence to equilibrium of non-coercive evolution equations based on entropy methods and getting quantitative descriptions of the rate, see \cite{Vil09} for the terminology. Our considerations are especially motivated by the result from \cite[Theo.~10]{DMS13} (or see \cite{DMS09}) in which hypocoercivity of the linear kinetic Fokker-Planck equation (with $\alpha=\beta=1$) associated to the Langevin dynamics on the Fokker-Planck Hilbert space
\begin{align*}
H_{\text{FP}}=L^2(F^{-1} \mathrm{d}x \otimes \mathrm{d}\omega),\quad F(x,\omega)
= \frac{1}{\sqrt{2\pi }^{d}} e^{-\Phi(x)} e^{-\frac{\omega^2}{2}}
\end{align*}
is proven. As noticed in \cite{DMS13}, the result from \cite[Theo.~10]{DMS13} is an important improvement to previous hypocoercivity results on the kinetic Fokker-Planck equation since it involves the first $L^2$-setting rather than a Sobolev space $H^1$-setting and moreover, requires weak assumptions on the underlying potential only. The statement \cite[Theo.~10]{DMS13} itself is an application of the abstract Hilbert space method from \cite[Sec.~1.3]{DMS13}. In this abstract method, it was the great idea of Dolbeault, Mouhot and Schmeiser to find a suitable entropy functional, which is equivalent to the underlying Hilbert space norm, for measuring the exponential decay to equilibrium. Consequently, the method \cite[Sec.~1.3]{DMS13} is simple and applies to a wide class of degenerate kinetic equations yielding conditions that are rather easy to verify in the applications of interest.

However, it is worth mentioning, that domain issues of the appearing operators are not included in the hypocoercivity setting from \cite[Sec.~1.3]{DMS13}. Thus computations are established algebraically and formally only therein.  And since \cite[Theo.~10]{DMS13} is an application of \cite[Sec.~1.3]{DMS13}, also the hypocoercivity theorem for the linear kinetic Fokker-Planck equation \cite[Theo.~10]{DMS13} is not yet complete. In order to give a complete elaboration, one needs a rigorous formulation of \cite[Sec.~1.3]{DMS13} first. The desired rigorous formulation of the method from \cite[Sec.~1.3]{DMS13} is given in \cite{GS12B}. The method from \cite{GS12B} contains the required domain issues and conditions for proving hypocoercivity need now only to be verified on a fixed operator core of the evolution operator. Moreover, the extended setting in \cite{GS12B} is suitably reformulated to incorporate also strongly continuous semigroups solving the Kolmogorov equation as an abstract Cauchy problem. In this way, it naturally applies to study the longtime behavior of the dynamics (in terms of transition kernels) of an SDE as will become clear below.

Summarizing, the aim of this article is to give a mathematical complete elaboration of the hypocoercivity statement
for the Langevin dynamics. For this purpose, we make use of our extended Kolmogorov hypocoercivity method from
article \cite{GS12B}. It further turns out that our elaboration requires then also an essential m-dissipativity
result for $(L,C_c^\infty(\mathbb{R}^{2d}))$ established in an article by the first author of this article in
\cite[Cor.~2.3]{CG10} as well as an essential selfadjointness result for
$(\Delta - \nabla \Phi \cdot \nabla,C_c^\infty(\mathbb{R}^d))$ from \cite[Theo.~3.1]{Wie85} or \cite[Theo.~7]{BKR97}.
However, these results from \cite{CG10} and \cite{BKR97} (or \cite{Wie85}) are not used in \cite{DMS13},
but are indispensable for providing a rigorous and complete elaboration. As an additional result,
we further compute the rate of convergence in dependence of the damping coefficient~$\alpha$. The resulting rate
obtained in \eqref{computation_rate_nu_2} confirms interesting phenomena, see Remark \ref{Rm_dependence_of_rate}.
The complete hypocoercivity theorem for the Langevin dynamics that can be achieved in our Kolmogorov setting now
reads as follows.

\begin{Thm} \label{Hypocoercivity_theorem_Langevin}
Let $d \in \mathbb{N}$ and $\alpha,\beta \in (0,\infty)$. Assume that $\Phi \colon \mathbb{R}^d \rightarrow \mathbb{R}$ is bounded from below, satisfies $\Phi \in C^{2}(\mathbb{R}^d)$ and that $e^{-\Phi} \mathrm{d}x$ is a probability measure on $(\mathbb{R}^d,\mathcal{B}(\mathbb{R}^d))$. Moreover, the measure $e^{-\Phi}\mathrm{d}x$ is assumed
to satisfy a Poincar\'e inequality\index{Poincar\'e inequality} of the form
\begin{align*}
\big\|\nabla f \big\|^2_{L^2(e^{- \Phi}\mathrm{d}x)} \geq \Lambda  \, \left\|\, f - \int_{\mathbb{R}^d} f \, e^{- \Phi}\mathrm{d}x \,\right\|^2_{L^2(e^{- \Phi}\mathrm{d}x)}
\end{align*}
for some $\Lambda \in (0,\infty)$ and all $f \in C_c^\infty(\mathbb{R}^d)$. Furthermore, assume that there exists
a constant $c < \infty$ such that
\begin{align*}
\left| \nabla^2 \Phi (x) \right| \leq c \left( 1+ \left| \nabla \Phi(x) \right|\right) \quad \mbox{for all}\quad  x \in \mathbb{R}^d.
\end{align*}
Then the Langevin operator $(L,C_c^\infty(\mathbb{R}^{2d}))$ is closable on $L^2(\mathbb{R}^{2d},\mu_{\Phi,\beta})$ and its closure $(L,D(L))$ generates a strongly continuous contraction semigroup $(T_t)_{t \geq 0}$. In particular, $(T_t)_{t \geq 0}$ provides a classical solution to the abstract Cauchy problem for $(L,D(L))$ in $L^2(\mathbb{R}^{2d},\mu_{\Phi,\beta})$. Moreover, $(T_t)_{t \geq 0}$ even admits a natural stochastic representation as the transition kernel of a diffusion process which provides a martingale (and even a weak) solution to the Langevin equation; see Remark \ref{Rm_stochastic_representation_Langevin_semigroup} for the details. Finally, for each $\nu_1 \in (1,\infty)$ there exists  $\nu_2 \in (0,\infty)$ such that
\begin{align*}
\left\|\,T_t g - \int_{\mathbb{R}^{2d}} g \, \mathrm{d}\mu_{\Phi,\beta} \,\right\|_{L^2(\mathbb{R}^{2d},\mu_{\Phi,\beta})} \leq \nu_1 e^{-\nu_2 \,t}  \left\|\,g - \int_{\mathbb{R}^{2d}} g \, \mathrm{d}\mu_{\Phi,\beta} \,\right\|_{L^2(\mathbb{R}^{2d},\mu_{\Phi,\beta})}
\end{align*}
for all $g \in {L^2(\mathbb{R}^{2d},\mu_{\Phi,\beta})}$ and all $t \geq 0$. Here $\nu_2$ can be specified as
\begin{align} \label{computation_rate_nu_2}
\nu_2 = \frac{\nu_1 - 1}{\nu_1} \,\frac{\alpha}{n_1 + n_2\, \alpha + n_3 \, \alpha^2}
\end{align}
and the constants $n_i \in (0,\infty)$, $i=1,\ldots,3$, only depend on the choice of $\Phi$ and $\beta$.
\end{Thm}

We remark that the conditions on $\Phi$ are mainly adapted from the original (algebraic) elaboration of the hypocoercivity theorem \cite[Theo.~10]{DMS13} in the \textit{dual} situation, i.e., for the linear kinetic Fokker-Planck equation with $\alpha=\beta=1$ on the Hilbert space $H_{\text{FP}}$. We mention that the conditions on $\Phi$ even originally occur in \cite[Theo.~35]{Vil09} where hypocoercivity of the linear kinetic Fokker-Planck equation is established in a suitable Sobolev norm.

As an important point, we emphasize that the usage of our Kolmogorov hypocoercivity method from \cite{GS12B} for our application is completely natural due to the stochastic representation for the Langevin semigroup $(T_t)_{t \geq 0}$ as stated in Theorem \ref{Hypocoercivity_theorem_Langevin}; see also  Remark \ref{Rm_stochastic_representation_Langevin_semigroup}. This stochastic representation result for the Langevin dynamics has been proven in two of the articles from the first named author of the underlying paper, see \cite{CG08} and \cite{CG10}. It is basically implied by using modern tools from the theory of (generalized) Dirichlet forms developed e.g.~in \cite{Fuk94}, \cite{MR92}, \cite{Tru00} or \cite{St99}. 
In the \textit{dual} Fokker-Planck situation, instead, we remark that there are no tools in literature available which yield corresponding representation formulas (in terms of probability densities) for the semigroup solving the abstract Fokker-Planck equation on the Hilbert space $H_{\text{FP}}$. Moreover, as explained in \cite[Part I, Sec.~7.4]{Vil09} or \cite[Sec.~2]{MV00}, considering the Fokker-Planck equation in an $L^2$-framework has even no real physical interpretation.

In relevant particle systems coming from Statistical Mechanics and Mathematical Physics, the potential in the Langevin equation is usually singular, e.g.~includes pair interactions of Lennard-Jones type. Discussing these cases is out of the scope of this article. However, it is an interesting problem of future research to establish Theorem \ref{Hypocoercivity_theorem_Langevin} also in such a situation. In this context, we refer to \cite{CG10} and \cite{GS13} in which ergodicity for the so-called $N$-particle Langevin dynamics with singular potentials is proven. The ergodicity method used therein shows up interesting analogies to the hypocoercivity method used in the underlying paper; see \cite{GS13} for details.

This article is organized as follows. In Section \ref{Hypocoercivity_method} we recapitulate our extended abstract Kolmogorov $L^2$-Hilbert space method presented in \cite{GS12B}. Afterwards, see Section \ref{Section_Hypocoercivity_Langevin_equation}, we give the desired complete elaboration of the $L^2$-Hilbert space hypocoercivity theorem for the Langevin dynamics by using our extended setting. Additionally, we calculate the rate of convergence in dependence of the damping coefficient $\alpha \in (0,\infty)$. The results of this article are obtained from the PhD thesis of the second named author; see \cite[Ch.~2]{Sti14}.

\section{The Hilbert space hypocoercivity method} \label{Hypocoercivity_method}

As described in the introduction, in this section we recapitulate the Hilbert space hypocoercivity method from \cite[Sec.~2]{GS12B}. It will be applied later on to establish hypocoercivity of the Langevin dynamics. The  method in \cite[Sec.~2]{GS12B} is a rigorous extension of the original hypocoercivity method from \cite[Sec.~1.3]{DMS13} in which domain issues are not yet included. Moreover, the formulation of the method in \cite[Sec.~2]{GS12B} is made for studying Kolmogorov (backward) evolution equations. Below, $\H$ always denotes a real Hilbert space with scalar product $(\cdot,\cdot)_\H$ and induced norm $\| \cdot \|$. All considered operators are assumed to be linear, defined on linear subspaces of $H$. An operator $(L,D(L))$ with domain $D(L)$ is also abbreviated by $L$. Basic knowledge from the theory of operator semigroups is assumed, see e.g.\cite{Paz83} and \cite{Gol85} for references. The upcoming data conditions (D) are assumed until the end of this section without mentioning this explicitly again.

\begin{Data} \label{Data}
$ $
\begin{itemize}
\item[(D1)] \textit{The Hilbert space.} Let $(E,\mathcal{F},\mu)$ be a probability space and define $\H$ to be $\H=L^2(E,\mu)$ equipped with the usual standard scalar product $\left(\cdot,\cdot\right)_\H$.\smallskip
\item[(D2)] \textit{The $C_0$-semigroup and its generator $L$.} $(L,D(L))$ is a linear operator on $\H$ generating a strongly continuous semigroup $(T_t)_{t \geq 0}$.\smallskip
\item[(D3)] \textit{Core property of $L$.}  Let $\D \subset D(L)$ be a dense subspace of $\H$ which is an operator core for $(L,D(L))$.\smallskip
\item[(D4)] \textit{Decomposition of $L$.} Let $(S,D(S))$ be symmetric and let $(A,D(A))$ be closed and antisymmetric on $\H$ such that $\D \subset D(S) \cap D(A)$ as well as $L_{|\D}=S-A$.\smallskip
\item[(D5)] \textit{Orthogonal projection.} Let $P\colon \H \to \H$ be an orthogonal projection which satisfies $P(\H) \subset D(S)$, $S P=0$ as well as $P(\D) \subset D(A)$, $AP (\D) \subset D(A)$. Moreover, we introduce
$P_{S}\colon \H \to \H$ as
\begin{align*}
P_S f \df P f + \left( f,1 \right)_\H, \quad f \in \H.
\end{align*}
\item[(D6)] \textit{The invariant measure.} Let $\mu$ be invariant for $(L,\D)$ in the sense that
\begin{align*}
\left( Lf,1 \right)_\H = \int_E Lf\, \mathrm{d}\mu=0 \quad \mbox{for all \quad $f \in \D$}.
\end{align*}
\item[(D7)] \textit{Semigroup conservativity.}  $1 \in D(L)$ and $L1=0$.
\end{itemize}
\end{Data}

Now the first three hypocoercivity conditions read as follows.

\begin{H1} \label{H1} (Algebraic relation) Assume that $P A P _{\,| \D}=0$.
\end{H1}

\begin{H2} \label{H2} (Microscopic coercivity) There exists $\Lambda_m > 0$ such that
\begin{align*}
-\left( Sf,f \right)_\H \geq \Lambda_m \, \|(I-P_S)f\|^2 \quad
\mbox{for all}\quad f \in \D.
\end{align*}
\end{H2}

\begin{H3} \label{H3} (Macroscopic coercivity) Define $(G,D)$ by $G=PA^2P$ on $D$. Assume that $(G,D)$ is essentially
selfadjoint on $H$ (or essentially m-dissipative on $H$
equi\-valently). Moreover, assume that there exists $\Lambda_M >
0$ such that
\begin{align} \label{eq_inequality_D3}
\|AP f\|^2 \geq \Lambda_M \|P f \|^2\quad \mbox{for all}\quad f
\in \D.
\end{align}
\end{H3}

In the hypocoercivity setting, one can introduce a suitable bounded linear operator $B$ on $H$ as follows. It is defined as the unique extension of $(B,D((AP)^*))$ to a continuous linear operator on $\H$ where
\begin{align*}
B\df(I+(AP)^* A P)^{-1}(AP)^* \quad \mbox{on}\quad D((AP)^*).
\end{align*}
Here $(AP,D(AP))$ is the linear operator $AP$ with domain
\begin{align*}
D(AP)=\{ f \in H ~|~ Pf \in D(A)\}
\end{align*}
and $((AP)^*,D((AP)^*))$ denotes its adjoint on $H$. Note that by von Neumann's theorem, the operator
\begin{align*}
I+(AP)^* A P \colon D((AP)^* A P) \to H
\end{align*}
with domain $D((AP)^* A P)=\{ f \in D(AP)~|~APf \in D((AP)^*)\}$
is bijective and admits a bounded inverse. Hence $B$ is indeed
well-defined on $D((AP)^*)$. For the fact that $B$ extends to a
bounded operator on $H$, consider the original references stated
above or see \cite[Theo.~5.1.9]{Ped89}. Now let $0 \leq
\varepsilon < 1$ and assume Condition (H1). The \textit{modified
entropy functional $\mathrm{H}_{\varepsilon}[\cdot]$}  is defined
by
\begin{align} \label{Df_H_epsilon}
\mathrm{H}_{\varepsilon}[f] \df \frac{1}{2} \|f\|^2 + \varepsilon \left(Bf,f\right)_\H,\quad f \in \H.
\end{align}
Then one obtains the following relation:
\begin{align} \label{equivalence_H_norm}
\frac{1-\varepsilon}{2} \|f\|^2 \leq \mathrm{H}_\varepsilon[f]
\leq \frac{1+\varepsilon}{2} \|f\|^2\quad \mbox{for all}\quad f
\in \H.
\end{align}

With the help of the previously defined operator $B$, one can introduce the last hypocoercivity condition, see next.

\begin{H4} \label{H4} (Boundedness of auxiliary operators) The operators $(BS,\D)$ and $(BA(I-P),\D)$ are bounded and there exists constants $c_1 < \infty$ and $c_2 < \infty$ such that
\begin{align*}
\|BSf\| \leq c_1 \,\|(I-P)f\|\quad \mbox{and}\quad \|BA(I-P)f\|
\leq c_2 \,\|(I-P)f\| \quad \mbox{for all}\quad  f \in D.
\end{align*}
\end{H4}

In order to verify (H4) for the Langevin dynamics later on, we need the following lemma.

\begin{Lm} \label{Lemma_sufficient_H4}

\begin{itemize}
\item[(i)]
Suppose that Condition (H1) holds. Assume $S(\D) \subset D(A)$ and assume that there exists $c_3 \in \mathbb{R}$ such that
\begin{align*}
PAS = c_3 \,PA \quad \mbox{on}\quad \D.
\end{align*}
Then the first inequality in (H4) holds with $c_1=
\displaystyle\frac{1}{2} |c_3|$.
\item[(ii)]
Assume that $(G,\D)$ is essentially selfadjoint and assume that there exists $c_4 < \infty$ such that
\begin{align} \label{eq_boundedness_BA*}
\|A^2Pf\| \leq c_4 \,\|g\| \quad \mbox{for all}\quad
g=(I-G)f,\quad f \in \D.
\end{align}
Then the second inequality is satisfied with $c_2=c_4$.
\end{itemize}
\end{Lm}

Then, assuming Conditions (D) and (H1)--(H4), the final
hypocoercivity theorem reads as follows.

\begin{Thm} \label{Thm_Hypocoercivity}
Assume that (D) and (H1)--(H4) holds. Then there exists strictly
positive constants $\kappa_1 < \infty$ and $\kappa_2< \infty$
which are computable in terms of $\Lambda_m,~\Lambda_M,~c_1$ and
$c_2$ such that for each $g \in \H$ we have
\begin{align*}
\left\|T_tg - \left(g,1\right)_\H\right\| \leq \kappa_1
e^{-\kappa_2 \,t}  \left\|g - \left(g,1\right)_\H\right\| \quad
\mbox{for all}\quad  t \geq 0.
\end{align*}
Here $(T_t)_{t \geq 0}$ denotes the $C_0$-semigroup introduced in (D2).
\end{Thm}

Later on, we are interested in deriving a dependence of $\kappa_1$ and $\kappa_2$ for the Langevin dynamics in terms of the damping parameter $\alpha$. For this purpose we need to recapitulate the proof of Theorem \ref{Thm_Hypocoercivity} from \cite[Theo.~2.18]{GS12B}; and see \cite{DMS13} for the original version of the proof.

\begin{proof}[Proof of Theorem \ref{Thm_Hypocoercivity}]
Let first $g \in D(L)$ and let $\mathrm{H}_\varepsilon[\cdot]$ be as in \eqref{Df_H_epsilon}. We define $(f_t)_{t \geq 0}$ as
\begin{align*}
f_t \df T_t g - \left( g, 1\right)_{\H} \quad \mbox{for all}\quad
t \geq 0.
\end{align*}
Now one needs to show that there exists a strictly positive constant $\kappa < \infty$ and a suitable $0 < \varepsilon <1$ (both independent of $g$) such that
\begin{align} \label{to_show_coercivity_of_D}
\mathrm{D}_\varepsilon [t]  := -\frac{d}{dt} \mathrm{H}_\varepsilon[f(t)]  \geq \kappa \,\|f(t)\|^2
\end{align}
holds for all $t \geq 0$. Indeed, assume the existence of such constants. By using \eqref{equivalence_H_norm} one obtains
\begin{align*}
\frac{d}{dt} \mathrm{H}_\varepsilon[f(t)] \leq -\frac{2 \kappa}{1
+ \varepsilon} \,\mathrm{H}_\varepsilon[f(t)] \quad \mbox{for
all}\quad t \geq 0.
\end{align*}
Gronwall's lemma and \eqref{equivalence_H_norm} then implies the
claim for $g \in D(L)$ with
$\kappa_1=\sqrt{\displaystyle\frac{1+\varepsilon}{1-\varepsilon}}$
and $\kappa_2 =\displaystyle\frac{\kappa}{1+\varepsilon}$. So, let
us verify the existence of the desired constants $\varepsilon$ and
$\kappa$ as required above. Therefore, the hypocoercivity
conditions (H1)--(H4) imply (see \cite[Sec.~1.3]{DMS13} or
\cite[Sec.~2]{GS12B}) that
\begin{align} \label{Choosing_the_constants_in_hypocoercivity}
\mathrm{D}_\varepsilon[t] &\geq \Lambda_m \|(I-P)f_t\|^2 +
\varepsilon \frac{\Lambda_M}{1+\Lambda_M} \|P f_t\|^2 -
\varepsilon (1+c_5)\, \| (I-P)f_t\| \| f_t\| \nonumber \\ & \geq
\left( \Lambda_m - \varepsilon (1+c_5) \left(1+\frac{1}{2 \delta}
\right) \right)\|(I-P)f_t\|^2 \\ & + \varepsilon \left(
\frac{\Lambda_M}{1+\Lambda_M} - (1+c_5) \frac{\delta}{2} \right)
\|P f_t\|^2, \nonumber
\end{align} 
where $c_5=c_1+c_2$ and $\delta > 0$ is arbitrary. Hence by fixing a suitable $\delta > 0$ and choosing $\varepsilon \in (0,1)$ small enough, observe that a constant $\kappa \in (0,\infty)$ can be found such that \eqref{to_show_coercivity_of_D} holds. Altogether, the statement is shown in case $g \in D(L)$. Note that the rate of convergence in terms of $\kappa_1$ and $\kappa_2$ is independent of $g \in D(L)$. Hence the claim follows by using denseness of $D(L)$ in $\H$.
\end{proof}

\begin{Rm}
For the Langevin dynamics later on, the constants $\kappa$ and $\varepsilon$ appearing in the proof of Theorem \ref{Thm_Hypocoercivity} are calculated in terms of the concrete constants $\Lambda_m$, $\Lambda_M$, $c_1$ and $c_2$ from the application. As seen in the previous proof, the choice of $\kappa$ and $\varepsilon$ determine the desired constants $\kappa_1$ and $\kappa_2$ explicitly.
\end{Rm}

\section{Hypocoercivity of the Langevin dynamics} \label{Section_Hypocoercivity_Langevin_equation}

As described in the introduction, the aim of this section is to prove exponential convergence to equilibrium
in our extended hypocoercivity framework
of the semigroup solving the abstract Kolmogorov equation corresponding to the classical Langevin
equation \eqref{Langevin_in _Rd_chapter_hypocoercivity}.  We remark that some specific calculations for
verifying (H1)--(H4) below are clearly simi\-lar to the associated original calculations for verifying Conditions (H1)--(H4) in the corresponding dual statement in the Fokker-Planck setting in \cite{DMS13}, see \cite[Theo.~10]{DMS13}. However, as already noticed before, in the proof of \cite[Theo.~10]{DMS13} domain issues are not taken into account. Crucial for our rigorous elaboration are the m-dissipativity and essential selfadjointness results derived in \cite[Cor.~2.3]{CG10} and \cite[Theo.~7]{BKR97} or \cite[Theo.~3.1]{Wie85}. We further remark that we additionally intend to compute the rate of convergence in dependence of the damping coefficient $\alpha \in (0,\infty)$ which is not done in \cite{DMS13}.

\subsection{The data conditions}

So, first of all we start introducing and verifying the conditions (D) from Section \ref{Hypocoercivity_method}. Recall that if $f$ is locally Lipschitz continuous on $\mathbb{R}^d$, then $f \in H^{1,\infty}_{\text{loc}}(\mathbb{R}^d)$ (see for instance \cite[Satz~8.5]{Alt06}). Moreover, $f$ is even differentiable $\d x$-a.e.~on $\mathbb{R}^d$ and the weak gradient $\nabla f$ coincides with the derivative of $f$ $\d x$-a.e.~on $\mathbb{R}^d$, see \cite[Theo.~6.15]{Hei01} and the proof of \cite[Theo.~6.17]{Hei01}. First we introduce the Hilbert space and our desired Kolmogorov backward operator associated to the Langevin equation \eqref{Langevin_in _Rd_chapter_hypocoercivity} under weak continuity assumptions on the potential $\Phi$. The following notations are used for the rest of this section without mention them again.

\begin{Df} \label{Df_operator_weak_ass_potential_Langevin_hypo} Let $d \in \mathbb{N}$ and $\alpha, \beta \in (0,\infty)$ in the Langevin equation \eqref{Langevin_in _Rd_chapter_hypocoercivity}. In the following, the first $d$ coordinates of $\mathbb{R}^{2d}$ are abbreviated with $x$ and the last $d$ coordinates by the variable $\omega$. The potential $\Phi \colon \mathbb{R}^d \to \mathbb{R}$ is assumed to be locally Lipschitz continuous and only depends on the position variable $x$. In the following, we fix a version of $\nabla \Phi = \nabla_x \Phi$. We introduce the measure space $(\mathbb{R}^{2d},\mathcal{B}(\mathbb{R}^{2d}),\mu_{\Phi,\beta})$ and the Hilbert space $\H$ as \index{$\mu_{\Phi,\beta}$}
\begin{align*}
\mu_{\Phi,\beta} \df e^{-\Phi(x)}\,\mathrm{d}x \otimes \nu_\beta,\quad \H \df L^2(\mathbb{R}^{2d},\mu_{\Phi,\beta}).
\end{align*}
Above $\nu_\beta$\index{$\nu_\beta$} denotes the normalized Gaussian measure on $\mathbb{R}^d$\index{Gaussian measure} with mean $0$ and covariance matrix $\beta^{-1} I$, see Section \ref{section_Introduction_Hypo}. Of course, we are only interested in potentials such that $\mu_{\Phi,\beta}$ is a finite measure. Thus w.l.o.g.~we assume $\mu_{\Phi,\beta}(\mathbb{R}^{2d})=1$ which equivalently means that $e^{-\Phi}\mathrm{d}x$ is a probability measure on $(\mathbb{R}^d,\mathcal{B}(\mathbb{R}^d))$. We introduce $\D$ as $\D \df C^\infty_c(\mathbb{R}^{2d})$ and the linear operators $(S,\D)$, $(A,\D)$ on the Hilbert space $\H$ by\index{Kolmogorov generator!Langevin dynamics}
\begin{align} \label{repr_S_A_L_Langevin_hypo}
A \df - \omega \cdot \nabla_x +  \frac{1}{\beta}\,\nabla_x \Phi
\cdot \nabla_\omega,\quad S \df - \alpha ~ \omega \cdot
\nabla_\omega +  \frac{\alpha}{\beta} \, \Delta_\omega \quad
\mbox{on}\quad \D.
\end{align}
Finally, the Langevin Kolmogorov operator $(L,D)$ is then defined by
\begin{align*}
L \df S-A \quad \mbox{on}\quad D.
\end{align*}
\end{Df}

Next, we introduce the desired projections $P$ and $P_S$.

\begin{Df} \label{Df_projections_P_PS_Langevin_hypo}
Assume the situation from Definition \ref{Df_operator_weak_ass_potential_Langevin_hypo}. Define $P_S\colon \H \to \H$ by
\begin{align*}
P_S  f \df \int_{\mathbb{R}^d} f\, \mathrm{d}\nu_\beta,\quad f \in \H.
\end{align*}
Here integration is understood w.r.t.~the $\omega$-coordinate. By using Fubini's theorem and the fact that $(\mathbb{R}^d,\mathcal{B}(\mathbb{R}^d),\nu_\beta)$ is a probability measure, one easily sees that $P_S$ is a well-defined orthogonal projection on $\H$ satisfying
\begin{align*}
P_S  f \in L^2(e^{-\Phi}\mathrm{d}x)\quad \mbox{and} \quad \| P_S  f\|_{L^2(e^{-\Phi}\mathrm{d}x)}=\| P_S f\|_{\H},\quad f \in \H.
\end{align*}
Here $L^2(e^{-\Phi}\mathrm{d}x)$ is canonically viewed as embedded in $L^2(\mu_{\Phi,\beta})$. Now $P\colon \H \to \H$ is given as
\begin{align*}
P f \df P_S   f - \left(f,1\right)_\H,\quad f \in \H.
\end{align*}
By using further that $\mu_{\Phi,\beta}(\mathbb{R}^{2d})=1$, one easily checks that $P$ is also an orthogonal projection fulfilling
\begin{align*}
P  f \in L^2(e^{-\Phi }\mathrm{d}x)\quad \mbox{and}\quad \| P  f\|_{L^2(e^{-\Phi}\mathrm{d}x)}=\| P f\|_{\H},\quad f \in \H.
\end{align*}
Finally, note that for each $f \in D$ the function $P_Sf$ admits a unique version from $C_c^\infty(\mathbb{R}^d)$. For notation convenience, we write
\begin{align*}
f_S \df P_S f \in C_c^\infty(\mathbb{R}^d),\quad f \in D.
\end{align*}
\end{Df}

Below we always make use of a suitable cut-off function as defined next. The choice of the cut-off function is standard, see e.g.~\cite[Prop.~5.5]{HN05}.

\begin{Df} \label{standard_cut_off_function} \index{standard cut-off function}
Let $k \in \mathbb{N}$. Choose some $\varphi \in C_c^\infty(\mathbb{R}^{k})$ such that $0 \leq \varphi \leq 1$, $\varphi =1$ on $B_1(0)$ and $\varphi=0$ outside $B_2(0)$.  Define
\begin{align*}
\varphi_n(z) \df \varphi(\frac{z}{n}) \quad \mbox{for each\quad $z
\in \mathbb{R}^{k}$,\quad  $n \in \mathbb{N}$}.
\end{align*}
Then there exists a constant $C < \infty$, independent of $n \in \mathbb{N}$, such that
\begin{align} \label{eq_boundedness_derivatives_varphi_n}
|\partial_{i} \varphi_n(z)| \leq  \frac{C}{n}, \quad
|\partial_{ij} \varphi_n(z )| \leq \frac{C}{n^2}\quad \mbox{for
all}\quad z \in \mathbb{R}^{k},\quad 1 \leq i,j \leq k.
\end{align}
Moreover, clearly $0 \leq \varphi_n \leq 1$ for all $n \in \mathbb{N}$ and $\varphi_n \to 1$ pointwisely on $\mathbb{R}^k$ as $n \to \infty$.
\end{Df}

The upcoming statement summarizes basic properties of the Langevin operator.

\begin{Lm} \label{Lm_properties_of_L_Langevin_hyp} Let $\Phi \colon \mathbb{R}^d \to\mathbb{R}$ be locally
Lipschitz continuous and let $(L,\D)$, $L=S-A$ on
$\D=C_c^\infty(\mathbb{R}^{2d})$,
$\H=L^2(\mathbb{R}^{2d},\mu_{\Phi,\beta})$ and the probability
measure $\mu_{\Phi,\beta}$ be as in
Definition~\ref{Df_operator_weak_ass_potential_Langevin_hypo}.
Then
\begin{itemize}
\item[(i)]
$(S,\D)$ is symmetric and nonpositive definite on $\H$.
\item[(ii)]
$(A,\D)$ is antisymmetric on $\H$.
\item[(iii)]
$\mu_{\Phi,\beta}$ is invariant for $(L,\D)$ in the sense that
\begin{align*}
\mu_{\Phi,\beta}(Lf) = \int_{\mathbb{R}^{2d}} Lf \,\mathrm{d}\mu_{\Phi,\beta} =0, \quad f \in D.
\end{align*}
\end{itemize}
Additionally, let $\nabla \Phi \in L^2(e^{-\Phi} \mathrm{d}x)$. Denote by $(L,D(L))$, $(S,D(S))$ and $(A,D(A))$ the closures of the dissipative operators $(L,D)$, $(S,D)$ and $(A,D)$ on $H$. Then
\begin{itemize}
\item[(iv)]
$P(\H) \subset D(S)$, $S P=0$ as well as $P(\D) \subset D(A)$ and $AP (\D) \subset D(A)$. Moreover, we have the natural formulas
\begin{align} \label{formula_AP_hypo_Langevin}
APf  = - \omega \cdot \nabla_x \,f_S,\quad f\in \D
\end{align}
as well as
\begin{align} \label{formula_AAP_hypo_Langevin}
A^2Pf  = \left(\omega, \nabla_x^2 \,f_S\, \omega \right)_{\text{euc}}  - \frac{1}{\beta} \, \nabla \Phi \cdot \nabla_x f_S, \quad f \in \D.
\end{align}
\item[(v)]
It holds $1 \in D(L)$ and $L1=0$.
\end{itemize}
\end{Lm}

\begin{proof} Properties (i)--(iii) can easily be verified using integration by parts, see for instance \cite[Lem.~4]{CG08} or \cite[Sec.~6.2]{Con11}. So, let us prove (iv) which contains calculations similar as performed in the proof of \cite[Lem.~3.7]{CG10}.

First let $f \in C_c^\infty(\mathbb{R}^d)$ and choose a sequence of cut-off functions $(\varphi_n)_{n \in \mathbb{N}}$ in $\mathbb{R}^d$ as in Definition \ref{standard_cut_off_function}. Define $f_n \in \D$, $n \in \mathbb{N}$, by
\begin{align} \label{f_n_first_choice_Lm_properties_of_L_Langevin_hyp}
f_n(x,\omega) \df f(x) \, \varphi_n(\omega),\quad (x,\omega) \in \mathbb{R}^{2d}.
\end{align}
Then by Lebesgue's dominated convergence theorem in combination with $|\omega|\in L^2(\nu_\beta)$ and the inequalities from \eqref{eq_boundedness_derivatives_varphi_n} we can infer that
\begin{align*}
Sf_n = \frac{\alpha}{\beta}\,f\, \Delta_\omega \,\varphi_n  -
\alpha\,f\, \omega \cdot \nabla_\omega \varphi_n \to 0 \quad
\mbox{with convergence in $\H$  as\quad $n \to \infty$.}
\end{align*}
This shows that $f \in D(S)$ and $Sf=0$ since $f_n \to f$ in $\H$ as $n \to \infty$ and $(S,D(S))$ is closed.

Now note that each element from the range of $P$ lies in $L^2(e^{-\Phi}\mathrm{d}x)$. So, choose an arbitrary $h \in L^2(e^{-\Phi}\mathrm{d}x)$. We have that $C_c^\infty(\mathbb{R}^d)$ is dense in $L^2(e^{-\Phi}\mathrm{d}x)$. Thus there exists $h_n \in C_c^\infty(\mathbb{R}^d)$, $n \in \mathbb{N}$, such that $h_n \to h$ in $L^2(e^{-\Phi}\mathrm{d}x)$ as $n \to \infty$. Now identify all $h_n$, $n \in \mathbb{N}$, and $h$ with elements from $\H$. By the previous consideration we have $h_n \in D(S)$ and $Sh_n=0$ for each $n \in \mathbb{N}$. Again from closedness of $(S,D(S))$ we can infer that $h \in D(S)$ and $Sh=0$. This shows $P(\H) \subset D(S)$ and $S P=0$.

Now let again $f \in C_c^\infty(\mathbb{R}^d)=P_S(\D)$ and define $(f_n)_{n \in \mathbb{N}}$ as in \eqref{f_n_first_choice_Lm_properties_of_L_Langevin_hyp}. Then dominated convergence implies
\begin{align*}
\omega \cdot \nabla_x f_n =  \varphi_n ~\omega \cdot \nabla_{x} f  \to \omega \cdot \nabla_{x} f,\quad \nabla \Phi \cdot \nabla_\omega f_n =  f ~\nabla \Phi \cdot \nabla_{\omega} \varphi_n  \to 0
\end{align*}
as $n \to \infty$ with convergence in $H$. Here we have used that $|\omega| \in L^2(\nu_\beta)$, the estimates from \eqref{eq_boundedness_derivatives_varphi_n} and $\nabla \Phi \in L^2(e^{-\Phi} \mathrm{d}x)$. Thus $f \in D(A)$ and we get $Af=-\omega \cdot \nabla_x f$. In order to show that $P(\D) \subset D(A)$ and to prove the first formula in (iv), it is left to show that $1 \in D(A)$ and $A1=0$. Therefore, use once more the closedness of $(A,D(A))$ and observe that the sequence $\psi_n(x,\omega) \df \varphi_n(x)$, $(x,\omega) \in \mathbb{R}^{2d}$, satisfies
\begin{align} \label{conservativity_of_A_Lm_properties_of_L_Langevin_hyp}
\psi_n \to 1, \quad A \psi_n = -\omega \cdot \nabla_x \psi_n  \to 0 \quad \mbox{with convergence in $\H$  as $n \to \infty$}
\end{align}
again due to  $|\omega| \in L^2(\nu_\beta)$, \eqref{eq_boundedness_derivatives_varphi_n} and by dominated convergence.

Next we show that $AP (\D) \subset D(A)$ and the second formula in (iv). Therefore, let $g$ be of the form $g=\omega_i \, f(x)$ where $f \in C_c^\infty(\mathbb{R}^{d})$. Here $\omega_i$ denotes the coordinate function $\mathbb{R}^{d} \ni \omega \mapsto \omega _i \in \mathbb{R}$ for some $1 \leq i \leq d$. Define $g_n$, $n \in \mathbb{N}$, by
\begin{align*}
g_n(x,\omega) \df \varphi_n(\omega) \, \omega_i \, f(x)\quad
\mbox{for\quad  $(x,\omega) \in \mathbb{R}^{2d}$}.
\end{align*}
Then again by dominated convergence in combination with $|\omega|,|\omega|^2 \in L^2(\nu_\beta)$,
\eqref{eq_boundedness_derivatives_varphi_n} and  $\nabla \Phi \in L^2(e^{-\Phi} \mathrm{d}x)$ we can infer that
\begin{align*}
\omega \cdot \nabla_x \,g_n = \varphi_n \, \omega_i ~\omega \cdot \nabla_x f \to \omega \cdot \nabla_x g
\end{align*}
as well as
\begin{align*}
\nabla \Phi \cdot \nabla_\omega g_n =f\,\omega_i\,\nabla\Phi \cdot \nabla_\omega \varphi_n + \varphi_n \, f \, \partial_{x_i} \Phi \to \nabla \Phi \cdot \nabla_\omega g
\end{align*}
with convergence in $H$ as $n \to \infty$ in each case. Thus by closedness of $(A,D(A))$ we conclude that each $g=\omega \cdot \nabla_x f$, $f \in C_c^\infty(\mathbb{R}^d)$, is an element from $D(A)$ and $A$ operates on $g$ in the natural way via the representation from $A$ as in \eqref{repr_S_A_L_Langevin_hypo}. Hence (iv) is shown.

Finally, we prove (v). However, this is now obvious. Therefore, let first $f \in C_c^\infty(\mathbb{R}^d)$ and $(f_n)_{n \in \mathbb{N}}$ be as in \eqref{f_n_first_choice_Lm_properties_of_L_Langevin_hyp} and let $(\psi_n)_{n \in \mathbb{N}}$ be as defined previously. Note that the calculations above also imply $f \in D(L)$ and
\begin{align*}
Lf=-Af= \omega \cdot \nabla_x f (x).
\end{align*}
In particular, as in \eqref{conservativity_of_A_Lm_properties_of_L_Langevin_hyp}, this identity yields $L \psi_n \to 0$ in $\H$ as $n \to \infty$ showing that $1 \in D(L)$ and $L1=0$ by closedness of $(L,D(L))$.
\end{proof}

Summarizing, besides (D3) all other data conditions are fulfilled. However, clearly (D3) is the hardest part and one has to prove essential m-dissipativity of the Langevin generator $(L,C_c^\infty(\mathbb{R}^{2d}))$ on $\H$. In case $\Phi \in C^\infty(\mathbb{R}^d)$ this is shown by Helffer and Nier in \cite[Prop.~5.5]{HN05} by using hypoellipticity techniques\index{hypoellipticity} and seems to be well-known to the community. In the article \cite{CG10}, essential m-dissipativity of $(L,C_c^\infty(\mathbb{R}^{2d}))$ on $H$ could even be established under more general assumptions on $\Phi$, see \cite[Cor.~2.3]{CG10}. More precisely, the result reads as follows.

\begin{Thm} \label{Thm_essent_m_dissipativity_Langevi_operator}\index{Langevin semigroup\\on $L^2(\mu_{\Phi,\beta})$}
Let $d \in \mathbb{N}$ and $\alpha,\beta \in (0,\infty)$. Assume that the potential $\Phi \colon \mathbb{R}^d \to \mathbb{R}$ is locally Lipschitz continuous and bounded from below. Then the generator of the Langevin dynamics $(L,C_c^\infty(\mathbb{R}^{2d}))$ from Definition \ref{Df_projections_P_PS_Langevin_hypo} is essentially m-dissipative on $H$. Thus its closure $(L,D(L))$ generates a strongly continuous contraction semigroup $(T_t)_{t \geq 0}$ on $H$.
\end{Thm}

\begin{Rm} \label{Rm_stochastic_representation_Langevin_semigroup}
In order to demonstrate that our Hilbert space hypocoercivity Kolmogorov setting is indeed natural, let us mention the following stochastic representation for the semigroup $(T_t)_{t \geq 0}$ associated with the closure of the Langevin generator $(L,C_c^\infty(\mathbb{R}^{2d}))$ on $\H$. Therefore, let the assumptions on $\Phi$ from Theorem \ref{Thm_essent_m_dissipativity_Langevi_operator} be satisfied. In \cite[Theo.~3]{CG08}, \cite[Theo.~2.5]{CG10} or \cite[Theo.~6.3.2]{Con11} combined with \cite[Lem.~2.2.8]{Con11}, it is shown that there exists a $\mu_{\Phi,\beta}$-tight Hunt process
\begin{align*}
\mathbf{M}=\left( \Omega, \mathcal{F}, (\mathcal{F}_t)_{t \geq 0},  (x_t,\omega_t)_{t \geq 0}, \mathbb{P}_{(x,\omega) \in \mathbb{R}^{2d}} \right)
\end{align*}
with infinite lifetime and continuous sample paths which is associated with $(T_t)_{t \geq 0}$ in the sense that $T_tf$, $t > 0$, is a $\mu_{\Phi,\beta}$-version of the transition semigroup
\begin{align*}
\mathbb{R}^{2d} \ni z \mapsto \mathbb{E}^z[f(x_t,\omega_t)]
\end{align*}
for any bounded $f\colon \mathbb{R}^{2d} \to \mathbb{R}$ with $f \in L^2(\mathbb{R}^{2d},\mu_\Phi)$. Moreover, $\mathbf{M}$ provides a martingale solution to the Langevin equation \eqref{Langevin_in _Rd_chapter_hypocoercivity} in the following sense: For quasi any starting point $(x,\omega) \in \mathbb{R}^{2d}$ the law $\mathbb{P}_{(x,\omega)}$ solves the martingale problem for $(L,C_c^2(\mathbb{R}^{2d}))$. Moreover, it even is a weak solution to the Langevin equation \eqref{Langevin_in _Rd_chapter_hypocoercivity}. For precise notations, we refer to the above mentioned references.

Summarizing, this shows the connection of the analytic hypocoercivity Kolmogorov approach with the original stochastic problem arised from SDE \eqref{Langevin_in _Rd_chapter_hypocoercivity}. We further remark that such stochastic representations can be established in general via using tools from the theory of (generalized) Dirichlet forms, see e.g.~\cite{Fuk80}, \cite{Fuk94}, \cite{MR92}, \cite{Roe99}, \cite{Tru00} or \cite{St99}.
\end{Rm}

\subsection{The hypocoercivity conditions}

Now we verify the hypocoercivity assumptions (H1)--(H4) for the
Langevin dynamics. Recall the notations from Definition
\ref{Df_operator_weak_ass_potential_Langevin_hypo}. First we
introduce the necessary conditions on the potential $\Phi$ that
are required below. Always let $\alpha,\beta \in (0,\infty)$.

\begin{Ass3} We need the following conditions.
\begin{itemize}
\item[(C1)] The potential $\Phi \colon \mathbb{R}^d \to \mathbb{R}$ is bounded from below, satisfies $\Phi \in C^{2}(\mathbb{R}^d)$ and $e^{-\Phi} \mathrm{d}x$ is a probability measure on $(\mathbb{R}^d,\mathcal{B}(\mathbb{R}^d))$.\smallskip
\item[(C2)]  The probability measure $e^{-\Phi}\mathrm{d}x$ satisfies a Poincar\'e inequality\index{Poincar\'e inequality} of the form
\begin{align*}
\left\|\nabla f \right\|^2_{L^2(e^{-\Phi}\mathrm{d}x)} \geq \Lambda  \, \left\| f - \left(f,1\right)_{L^2(e^{-\Phi}\mathrm{d}x)} \,\right\|^2_{L^2(e^{-\Phi}\mathrm{d}x)}
\end{align*}
for some $\Lambda \in (0,\infty)$ and all $f \in C_c^\infty(\mathbb{R}^d)$.\smallskip
\item[(C3)]  There exists a constant $c < \infty$ such that
\begin{align*}
\left| \nabla^2 \Phi (x) \right| \leq c \left( 1+ \left| \nabla
\Phi(x) \right|\right) \quad  \mbox{for all}\quad  x \in
\mathbb{R}^d.
\end{align*}
\end{itemize}
\end{Ass3}

Condition (C2) is necessary to show (H3) and in order to prove (H4) we essentially need Conditions (C2) and (C3). We remark that Condition (C3) together with the property that $e^{-\Phi} \mathrm{d}x$ is a probability measure indeed implies that $\nabla \Phi \in L^2(e^{-\Phi}\mathrm{d}x)$, see \cite[Lem.~A.24]{Vil09}. Moreover,
the Poincar\'e inequality is satisfied for instance if
\begin{align*}
\frac{|\nabla \Phi(x)|^2}{2} - \Delta \Phi(x) \stackrel{|x| \rightarrow \infty}{\longrightarrow} +\infty,
\end{align*}
see e.g.~\cite{BBCG08} or \cite[A.~19]{Vil09}. For further
references on  Poincar\'e inequalities, see \cite{Wan99}. Examples
for potentials fulfilling Conditions (C1)--(C3) (after
normalization) are e.g.~$\Phi=\|\cdot\|^p$, $p=2,4,$ or $p \geq
6$, since in these cases a Poincar\'e inequality is satisfied; see
e.g.~\cite{Wan99} or~\cite{RW01}.

Now let us start with the verification of (H1).

\begin{Pp} \label{Pp_H1_Langevin}
Let $\Phi$ be as in Definition \ref{Df_operator_weak_ass_potential_Langevin_hypo}. This means that $\Phi \colon \mathbb{R}^d \to \mathbb{R}$ is locally Lipschitz continuous and $e^{-\Phi} \mathrm{d}x$ is a probability measure on $(\mathbb{R}^d,\mathcal{B}(\mathbb{R}^d))$. Further assume $\nabla \Phi \in L^2(e^{-\Phi}\mathrm{d}x)$. Then (H1) holds.
\end{Pp}

\begin{proof}
Let $f \in \D$. Then by the first formula from Lemma \ref{Lm_properties_of_L_Langevin_hyp}\,(iv) we have
\begin{align*}
A P f= - \omega \cdot \nabla_x f_S,\quad \mbox{where} \quad
f_S=P_S f \in C_c^\infty(\mathbb{R}^d).
\end{align*}
Thus we conclude $P_S A Pf =0$ since
\begin{align*}
\int_\mathbb{\mathbb{R}^d} \left(\omega, z\right)_{\text{euc}} \,
\mathrm{d}\nu_\beta=0 \quad \mbox{for all}\quad z \in
\mathbb{R}^d.
\end{align*}
Then also
\begin{align*}
\left(A Pf,1\right)_\H=\left(P_S A Pf,1\right)_{L^2(e^{-\Phi} \mathrm{d}x)}=0.
\end{align*}
Hence $P A P = 0$ on $\D$ as desired.
\end{proof}

\begin{Pp} \label{Pp_H2_Langevin}
Let $\Phi$ be as in Definition \ref{Df_operator_weak_ass_potential_Langevin_hypo}. Then Condition (H2) is satisfied with $\Lambda_m = \alpha$.
\end{Pp}

\begin{proof}
The Poincar\'e  inequality for the Gaussian measure, see
\cite{Be89}, easily implies
\begin{align*}
\big\|\nabla_\omega f \big\|^2_{L^2(\nu_\beta)} \geq \beta \,
\left\| f- \int_{\mathbb{R}^d} f(\omega) \,
\mathrm{d}\nu_\beta(\omega) \right\|^2_{L^2(\nu_\beta)} \quad
\mbox{for all}\quad f \in C_c^\infty(\mathbb{R}^{d}).
\end{align*}
In other words, we obtain
\begin{align*}
-\left(Sf,f\right)_H = \frac{\alpha}{\beta} \,\big\|\nabla_\omega
f \big\|^2_{H} \geq \alpha \, \| f- P_Sf \|^2_{H} \quad \mbox{for
each}\quad f \in D.
\end{align*}
The claim follows.
\end{proof}

Next, we calculate the operator $G \df PA^2P$ on $D$. Below we need Condition (C1). Let us therefore already assume it and let $\nabla \Phi \in L^2(e^{-\Phi}\mathrm{d}x)$. By the second formula from Lemma \ref{Lm_properties_of_L_Langevin_hyp}\,(iv) we obtain
\begin{align*}
P_SA^2Pf = \frac{1}{\beta}\,\Delta_x f_S - \frac{1}{\beta} \, \nabla \Phi \cdot \nabla_x f_S, \quad f \in D.
\end{align*}
For the moment, consider the operator $(T,C_c^\infty(\mathbb{R}^d))$ on the Hilbert space $L^2(e^{-\Phi}\mathrm{d}x)$ defined by $T= \Delta_x -  \nabla_x \Phi \cdot \nabla_x$ on $C_c^\infty(\mathbb{R}^d)$. Then for each $h \in C^\infty_c(\mathbb{R}^d)$ and $g \in C^\infty(\mathbb{R}^d)$ it holds using integration by parts
\begin{align*}
\left( Th, g \right)_{L^2(e^{-\Phi} \mathrm{d}x)} = -  \int_{\mathbb{R}^d} \nabla h \cdot \nabla g ~e^{-\Phi} \mathrm{d}x.
\end{align*}
In particular, we have $\left( Th, 1\right)_{L^2(e^{-\Phi} \mathrm{d}x)}=0$. Thus, since $f_S \in C^\infty_c(\mathbb{R}^d)$, we conclude
\begin{align*}
\left(A^2 P f, 1\right)_\H= \left(P_S A^2 P f, 1\right)_{L^2(e^{-\Phi} \mathrm{d}x)}=\frac{1}{\beta}\,\left(T f_S,1 \right)_{L^2(e^{-\Phi} \mathrm{d}x)}=0.
\end{align*}
So, for each $f \in \D$, we obtain the formula
\begin{align} \label{formula_PAAP_Langevin}
P A^2 P f = \frac{1}{\beta} \left( \Delta f_S -  \nabla \Phi \cdot \nabla f_S \right).
\end{align}
In order to verify (H3), we need the upcoming statement first.

\begin{Pp} \label{density_of_PAAPD_hypo_Langevin}
Assume that the potential $\Phi\colon \mathbb{R}^d \to \mathbb{R}$ fulfills Condition (C1) and assume $\nabla \Phi \in L^2(e^{-\Phi}\mathrm{d}x)$. Then $(I-P A^2 P)(\D)$ is dense in $\H$. This means that $(G,\D)$ is essentially m-dissipative on $\H$, hence essentially selfadjoint on $\H$.
\end{Pp}

\begin{proof} First recall that for densely defined, symmetric and dissipative linear operators on a Hilbert space, the property of being essential m-dissipative is equivalent to essential selfadjointness. Now let $(T,C_c^\infty(\mathbb{R}^d))$ be as defined above. By \cite[Theo.~7]{BKR97} or \cite[Theo.~3.1]{Wie85} our assumptions in particular imply that $(T,C_c^\infty(\mathbb{R}^d))$ is essentially selfadjoint on $L^2(e^{-\Phi} \mathrm{d}x)$. Hence $(T,C_c^\infty(\mathbb{R}^d))$ is also essentially m-dissipative on $L^2(e^{-\Phi} \mathrm{d}x)$. Now let $g \in \H$ such that
\begin{align} \label{eq_verification_core_H3_hypo_Langevin}
((I-G)f,g)_\H=0 \quad \mbox{for all}\quad  f \in \D.
\end{align}
We have to show that $g=0$. Choose $f \in C_c^\infty(\mathbb{R}^d)$ and let $(f_n)_{n \in \mathbb{N}}$, $(\varphi_n)_{n \in \mathbb{N}}$ be as in \eqref{f_n_first_choice_Lm_properties_of_L_Langevin_hyp}. Then Identity \eqref{eq_verification_core_H3_hypo_Langevin} implies
\begin{align*}
0=((I-G)f_n,g)_{\H} = (\varphi_n f,g)_{\H} - \frac{1}{\beta}\,\|\varphi_n\|_{L^1(\nu_\beta)} (Tf,g)_{\H} \to (f,g)_{\H} - \frac{1}{\beta}\,(Tf,g)_{\H}
\end{align*}
as $n \to \infty$ by dominated convergence. Hence
\begin{align*}
((\beta I-T)f,P_S g)_{L^2(e^{-\Phi} \mathrm{d}x)}=((\beta
I-T)f,g)_{\H}=0 \quad \mbox{for all}\quad f \in
C_c^\infty(\mathbb{R}^d).
\end{align*}
Thus $P_S g=0$ in $L^2(e^{-\Phi} \mathrm{d}x)$ since $(\beta I-T)(C_c^\infty(\mathbb{R}^d))$ is dense in $L^2(e^{-\Phi} \mathrm{d}x)$. So, for each $f \in \D$ we can infer that
\begin{align*}
(Gf,g)_\H=\frac{1}{\beta} \, \left(Tf_S, P_S g\right)_{L^2(e^{-\Phi} \mathrm{d}x)}=0.
\end{align*}
Consequently, \eqref{eq_verification_core_H3_hypo_Langevin} yields $(f,g)_\H=0$ for each $f \in \D$. Hence $g=0$ as desired.
\end{proof}

Now we prove (H3).

\begin{Pp} \label{Pp_H3_Langevin}
Assume that $\Phi \colon \mathbb{R}^d \to \mathbb{R}$ satisfies
(C1) and (C2) and assume that $\nabla \Phi \in
L^2(e^{-\Phi}\mathrm{d}x)$. Then (H3) holds where $\Lambda_M =
\displaystyle\frac{\Lambda}{\beta}$.
\end{Pp}

\begin{proof}
Let $f \in \D$. By the Poincar\'e inequality for the probability measure $e^{-\Phi}\mathrm{d}x$ from (C2) we have
\begin{align*}
\| A P f \|^2_H &= \int_{\mathbb{R}^d} \int_{\mathbb{R}^d} \left( \omega \cdot \nabla_x f_S \right)^2 \,e^{-\Phi} \mathrm{d}\nu_\beta(\omega) \,\mathrm{d}x = \frac{1}{\beta} \int_{\mathbb{R}^d} \left| \nabla_x f_S \right|^2 \, e^{-\Phi} \mathrm{d}x \\
& \geq \frac{\Lambda}{\beta} \int_{\mathbb{R}^d} \left(f_S - \int f_S \,e^{-\Phi}  \mathrm{d}x \right)^2 e^{-\Phi} \mathrm{d}x = \frac{\Lambda}{\beta}\,\left\| P_S f - \left(f,1\right)_\H\right\|^2_H .
\end{align*}
So, Inequality \eqref{eq_inequality_D3} is fulfilled for all elements from $\D$. Together with Proposition \ref{density_of_PAAPD_hypo_Langevin}, Condition (H3) indeed follows.
\end{proof}

It is left to verify Condition (H4). Therefore, we need an
elliptic \textit{a priori} estimates from Dolbeault, Mouhot and
Schmeiser (see \cite{DMS13}) which especially requires all
Conditions (C1)--(C3) from above.

\begin{Pp} \label{Pp_H4_Langevin} Assume that $\Phi \colon \mathbb{R}^d \to \mathbb{R}$ satisfies (C1),(C2) and (C3).
Then Condition (H4) is satisfied. Moreover, the constants therein
are given by $c_1= \displaystyle\frac{1}{2} \alpha$ and $c_2 =
c_{\Phi,\beta}$ where $c_{\Phi,\beta} \in [0,\infty)$ depends on
the choice of $\Phi$ and $\beta$.
\end{Pp}

\begin{proof}
For the verification of (H4) we aim to apply Lemma \ref{Lemma_sufficient_H4}. First note that $S(\D) \subset \D$. We show that $PAS=\alpha\, PA$ on $\D$. This is clearly equivalent to
\begin{align*}
\left( Sg , APf \right)_{\H}=\alpha \,\left( g, APf \right)_{\H}
\quad \mbox{for all}\quad f,g\in \D.
\end{align*}
Indeed, the latter identity holds since
\begin{align*}
\left( Sg , APf \right)_{\H} &= \int_{\mathbb{R}^{2d}}
\left(\frac{\alpha}{\beta} \, \Delta_\omega g - \alpha ~\omega
\cdot \nabla_\omega g\right) \, \omega \cdot \nabla_x f_S ~
\mathrm{d}\mu_{\Phi,\beta} \\ &=\int_{\mathbb{R}^{2d}} g
\left(\frac{\alpha}{\beta} \,\Delta_\omega - \alpha~\omega \cdot
\nabla_\omega\right) \left( \omega \cdot \nabla_x f_S \right) \,
\mathrm{d}\mu_{\Phi,\beta}\\ &= - \alpha \int_{\mathbb{R}^{2d}} g
~ \omega \cdot \nabla_x f_S  \, \mathrm{d}\mu_{\Phi,\beta},
\end{align*}
where integration by parts has been used. Thus $PAS=\alpha \, PA$ on $\D$ and the first part of (H4) is satisfied by Lemma \ref{Lemma_sufficient_H4}\,(i). We prove the second part of (H4). Therefore, let $g \in \H$ be of the form $g=(I-PA^2P)f$ for some $f \in \D$. The second formula from Lemma \ref{Lm_properties_of_L_Langevin_hyp}\,(iv) implies
\begin{align} \label{estimation_BA*_Langevin}
\| A^2 f_P \| &\leq \left\| \left| \omega \right|^2
\right\|_{L^2(\nu_\beta)} \left\|  |\nabla_x^2  f_P|
\right\|_{L^2(e^{-\Phi}\mathrm{d}x)} +  \frac{1}{\beta} \, \left\|
\left|\nabla_x \Phi \right| \left|\nabla_x f_P\right|
\right\|_{L^2(e^{-\Phi}\mathrm{d}x)},
\end{align}
where $f_P \df f_S-\left(f_S,1\right)_{L^2(e^{-\Phi} \mathrm{d}x)} = Pf$ with $f_S \in C_c^\infty(\mathbb{R}^d)$. Now due to Identity \eqref{formula_PAAP_Langevin} note that $f_P$  solves the elliptic equation
\begin{align*}
f_P -  \frac{1}{\beta} \left(\Delta f_P - \nabla \Phi \cdot \nabla
f_P\right) =Pg \quad \mbox{in}\quad L^2(e^{-\Phi} \mathrm{d}x).
\end{align*}
By applying the elliptic \textit{a priori} estimates of Dolbeault, Mouhot and Schmeiser from \cite[Sec.~2, Eq.~(2.2), Lem.~8]{DMS13} (or see \cite[Appendix, Sec.~5.1]{GS12B} for corresponding proofs including domain issues) to the right hand side of Inequality \eqref{estimation_BA*_Langevin} we conclude
\begin{align*}
\|(BA)^*g\|_{\H}  &\leq c_{\Phi,\beta} \,\| Pg\|_{L^2(e^{-\Phi} \mathrm{d}x)} \leq c_{\Phi,\beta} \,\|g\|_\H
\end{align*}
for a constant $c_{\Phi,\beta} < \infty$ independent of $g$ and only depending on the choice of $\Phi$ and $\beta$.
Note that the \textit{a priori} estimates require Conditions (C1)--(C3). Finally, apply
Lemma~\ref{Lemma_sufficient_H4}\,(ii) to finish the proof.
\end{proof}

Altogether, we are are able to verify Theorem \ref{Hypocoercivity_theorem_Langevin}.

\begin{proof}[Proof of Theorem \ref{Hypocoercivity_theorem_Langevin}]
Collecting all results from the whole section, Theorem \ref{Thm_Hypocoercivity} implies the statement. Indeed, the hypocoercivity data conditions are fulfilled by Lemma \ref{Lm_properties_of_L_Langevin_hyp} and Theorem \ref{Thm_essent_m_dissipativity_Langevi_operator}. The hypocoercivity conditions (H1) up to (H4) are fulfilled due to Proposition \ref{Pp_H1_Langevin}, Proposition \ref{Pp_H2_Langevin}, Proposition \ref{Pp_H3_Langevin} and Proposition \ref{Pp_H4_Langevin}.

It is left to compute the rate of convergence in dependence of $\alpha$ as claimed in the statement. Therefore, we go back into the proof of Theorem \ref{Thm_Hypocoercivity}, modify the latter and try to choose the constants $\delta \in (0,\infty)$, $\varepsilon \in (0,1)$ and $\kappa \in (0,\infty)$ explicitly therein. The following calculations are analogous to (and basically taken from) the ones of the proof of Theorem 1 in \cite[Sec.~3.4]{DKMS11}. In \cite{DKMS11} namely, the hypocoercivity strategy from \cite{DMS13} is applied to the so-called two-dimensional fiber lay-down model with emphasis on calculating the rate of convergence in dependence of the so-called noise amplitude; So, below we use the notations introduced in the proof of Theorem \ref{Thm_Hypocoercivity} with our specific values for $\Lambda_m$, $\Lambda_M$, $c_1$, $c_2$ and we follow \cite[Sec.~3.4]{DKMS11}. We set
\begin{align*}
\delta \df \frac{\Lambda}{\beta + \Lambda} \, \frac{1}{1+c_{\Phi,\beta} + \frac{\alpha}{2}}.
\end{align*}
Now the coefficients of the right hand side of \eqref{Choosing_the_constants_in_hypocoercivity} can be written as $\alpha - \varepsilon \, {r}_{\Phi,\beta}(\alpha)$ and $\varepsilon \, {s}_{\Phi,\beta}$ where
\begin{align*}
r_{\Phi,\beta}(\alpha) \df \big(1+c_{\Phi,\beta} + \frac{\alpha}{2}\big) \big(1+\frac{\beta + \Lambda}{2\, \Lambda}
\big(1+c_{\Phi,\beta} + \frac{\alpha}{2}\big)\big),\quad s_{\Phi,\beta} \df \frac{1}{2} \, \frac{\Lambda}{\beta + \Lambda}.
\end{align*}
Here $\varepsilon =: \varepsilon_{\Phi,\beta}(\alpha) \in (0,1)$ needs still to be determined. Note that ${r}_{\Phi,\beta}(\alpha) + s_{\Phi,\beta}$ is of the form
\begin{align*}
{r}_{\Phi,\beta}(\alpha) + s_{\Phi,\beta} = a_1 + a_2 \, \alpha +
a_3 \alpha^2,
\end{align*}
where all $a_i \in (0,\infty)$, $i=1,\ldots,3$, depend on the choice of $\Phi$ and $\beta$. One defines
\begin{align} \label{Definition_og_overline_varepsilon}
\overline{\varepsilon}_{\Phi,\beta} (\alpha) \df
\frac{\alpha}{r_{\Phi,\beta}(\alpha) +
s_{\Phi,\beta}}=\frac{\alpha}{a_1 + a_2\, \alpha + a_3 \,
\alpha^2}.
\end{align}
Note that $\overline{\varepsilon}_{\Phi,\beta}(\alpha)$ is in general not the right choice for $\varepsilon$ since possibly $\overline{\varepsilon}_{\Phi,\beta}(\alpha) \geq 1$. Now let $\upsilon > 0$ be arbitrary. Define
\begin{align*}
\varepsilon \df \frac{\upsilon}{1+ \upsilon} \,
\frac{\overline{\varepsilon}_{\Phi,\beta}(\alpha)}
{\overline{\varepsilon}_{\Phi,\beta,\max}} \quad \mbox{with}\quad
\overline{\varepsilon}_{\Phi,\beta,\max} \df \max\{ 1,
\sup_{\alpha > 0} \overline{\varepsilon}_{\Phi,\beta}(\alpha)\}.
\end{align*}
Now really $0 < \varepsilon < 1$ and note that $\overline{\varepsilon}_{\Phi,\beta,\max}$ is well-defined due to \eqref{Definition_og_overline_varepsilon}. Then
\begin{align*}
\varepsilon \, r_{\Phi,\beta}(\alpha) + \, \varepsilon \, s_{\Phi,\beta} =\frac{\upsilon}{1+ \upsilon} \, \frac{\alpha}{\overline{\varepsilon}_{\Phi,\beta,\max}} \leq \alpha.
\end{align*}
Hence we get the estimation
\begin{align*}
\alpha - \varepsilon \,r_{\Phi,\beta}(\alpha) \geq \varepsilon \,
s_{\Phi,\beta} = \frac{\upsilon}{1+ \upsilon}
\,\frac{2\,\alpha}{n_1 + n_2\, \alpha + n_3 \, \alpha^2} =: \kappa
,
\end{align*}
where all $n_i \in (0,\infty)$ depend on $\Phi$ and $\beta$ and are given by
\begin{align*}
n_i\df 2\,\frac{\overline{\varepsilon}_{\Phi,\beta,\max}}{
s_{\Phi,\beta}}~a_i \quad \mbox{for each}\quad i=1,\ldots,3.
\end{align*}
Summarizing, the desired constant $\kappa \in (0,\infty)$ as required in the proof of Theorem \ref{Thm_Hypocoercivity} is found. From the proof of Theorem \ref{Thm_Hypocoercivity} we can infer that
\begin{align*}
\left\|\,T_t g - \int_{\mathbb{R}^{2d}} g \, \mathrm{d}\mu_{\Phi,\beta} \,\right\|_{L^2(\mathbb{R}^{2d},\mu_{\Phi,\beta})} \leq \kappa_1 e^{-\kappa_2 \,t}  \left\|\,g - \int_{\mathbb{R}^{2d}} g \, \mathrm{d}\mu_{\Phi,\beta} \,\right\|_{L^2(\mathbb{R}^{2d},\mu_{\Phi,\beta})}
\end{align*}
for each $g \in L^2(\mu_{\Phi,\beta})$ and each $t \geq 0$. Here
$\kappa_1 =
\sqrt{\displaystyle\frac{1+\varepsilon}{1-\varepsilon}}$ and
$\kappa_2 =\displaystyle\frac{\kappa}{1+\varepsilon}$. Finally, it
is easily verified that $\displaystyle
\frac{1+\varepsilon}{1-\varepsilon} \leq (1+\upsilon)^2$ and
$\kappa_2 \geq \displaystyle\frac{1}{2}\,\kappa$. So, via setting
\begin{align*}
\nu_1 \df 1+ \upsilon\quad \mbox{and} \quad \nu_2 \df
\displaystyle\frac{1}{2}\kappa
\end{align*}
the concrete rate of convergence claimed in the theorem is shown.
\end{proof}

Finally, we conclude with a remark as in \cite[Rem.~2.12]{GS13}
and \cite[Rem.~3.18]{GS13}.

\begin{Rm}\label{Rm_dependence_of_rate}
({\it{i}}).~The rate of convergence in dependence of $\alpha$ is
expected by the following heuristic considerations. Observe that
for small values of $\alpha$ close to zero one has a bad or very
slow decay towards $\mu_{\Phi,\beta}$ since the dynamics nearly
behaves deterministic in this situation. Vice versa, in a large
damping regime, the $(x_t)_{t \geq 0}$ process can be described
approximately by the overdamped Langevin dynamics, see
\cite[Sec.~2.2.4]{LRS10}. The scaling
\begin{align*}
\overline{t} = \frac{t}{\alpha}, \quad
\overline{x}_{\overline{t}}=x_t, \quad \overline{W}_{\overline{t}}
= \frac{1}{\sqrt{\alpha}} W_t, \quad
\overline{\omega}_{\overline{t}} = \alpha \, \omega_t,\quad
\overline{\Phi}(\overline{x}) = \Phi(x)
\end{align*}
formally yields
\begin{align*}
\d \overline{x}_{\overline{t}} =
\overline{\omega}_{\overline{t}}\, \d \overline{t} ,\quad
\frac{1}{\alpha^2} \d \overline{\omega}_{\overline{t}} = -
\overline{\omega}_{\overline{t}} \, \d \overline{t} -
\frac{1}{\beta} \, \nabla
\overline{\Phi}(\overline{x}_{\overline{t}})\,\d \overline{t} +
\sqrt{\frac{2}{\beta}}\, \d \overline{W}_{\overline{t}}.
\end{align*}
Thus $\displaystyle\frac{1}{\alpha^2} \d
\overline{\omega}_{\overline{t}} \to 0$ as $\alpha \uparrow
\infty$. So, setting $\displaystyle\frac{1}{\alpha^2} \d
\overline{\omega}_{\overline{t}} = 0$ for $\alpha$ large, solving
the equation w.r.t.~$\d \overline{x}_{\overline{t}} =
\overline{\omega}_{\overline{t}} \, \d \overline{t} $ and
rescaling yields the SDE in $\mathbb{R}^d$ given as
\begin{align} \label{overdamped_Langevin_Highfriction}
\mathrm{d}x_t = -\frac{1}{\alpha\,\beta} \, \nabla \Phi(x_t) \,
\mathrm{d}t + \sqrt{\frac{2}{\alpha \, \beta}}\, \mathrm{d}W_t.
\end{align}
with formal generator $L^{\text{ov}}=
\displaystyle\frac{1}{\alpha\,\beta}\Delta
-\displaystyle\frac{1}{\alpha\,\beta} \nabla \Phi \cdot \nabla$.
If $\Phi$ fulfills e.g.~(C1) and (C2) with $\Lambda
> 0$ the constant from the Poincar\'e inequality, it well-known
(and easy to verify) that the s.c.c.s.~$(S_t)_{t \geq 0}$ in
$L^2(e^{-\Phi}\d x)$ associated with the closure of
$(L^{\text{ov}},C_c^\infty(\mathbb{R}^d))$ satisfies
\begin{align*}
\left\| S_t f - \int f \, e^{-\Phi}
\mathrm{d}x\right\|_{L^2(e^{-\Phi}\mathrm{d}x)} \leq
e^{-\frac{\Lambda}{\alpha \beta} t} \left\|  f - \int f \,
e^{-\Phi} \mathrm{d}x\right\|_{L^2(e^{-\Phi}\mathrm{d}x)}
\end{align*}
for each $f \in L^2(e^{-\Phi}\mathrm{d}x)$. Altogether, the
convergence rate for the Langevin dynamics is expected to become
as worse as possible when $\alpha \uparrow \infty$. So, we see
that these phenomena on the convergence to equilibrium in
dependence of $\alpha > 0$ are rigorously proven and confirmed by
Theorem \ref{Hypocoercivity_theorem_Langevin} from the
introduction. Compare with \cite[Theo.~2.11]{GS13} where the same
qualitative convergence behavior for the Langevin dynamics in
dependence of $\alpha$ in an ergodicity setting is shown.

({\it{ii}}). As seen in this article, our Hilbert space
hypocoercivity setting can successfully be applied to the
classical degenerate Langevin dynamics. Another interesting
problem is the application of the hypocoercivity setting to
investigate the longtime behavior of the manifold-valued version
of the degenerate Langevin equation. This generalized version of
the Langevin equation is derived e.g.~in \cite{GS12} and in
\cite{LRS12} (where it is called the constrained Langevin
dynamics). The interest for studying the longtime behavior and
establish hypocoercivity of the geometric version of the Langevin
equation arised in \cite[Prop.~3.2]{LRS12} where an ergodic
statement for the constrained Langevin dynamics is outlined
without convergence rate.
\end{Rm}

\end{document}